\documentclass{article}
\usepackage{amssymb, amsthm, amsmath}
\usepackage{hyperref}
\newtheorem{thm}{Theorem}[section]
\newtheorem{lem}[thm]{Lemma}

\theoremstyle{definition}
\newtheorem{df}[thm]{Definition}
\theoremstyle{remark}

\numberwithin{equation}{section}
\newcommand{\restrict}{\upharpoonright}
\newcommand{\la}{\langle}\newcommand{\ra}{\rangle}
\newcommand{\mc}{\mathcal}

\newcommand{\inter}{\cap}

\renewcommand{\to}{\rightarrow}
\newcommand{\Iff}{\Leftrightarrow}

\newcommand{\Implies}{\Rightarrow}

\begin{document}
	\title{Low for random reals and positive-measure domination}
	\author{Bj\o rn Kjos-Hanssen\footnote{The author thanks the Institute for Mathematical Sciences of the National University of Singapore for support during preparation of this manuscript at the \emph{Computational Prospects of Infinity} conference in Summer 2005. The author also thanks Denis R. Hirschfeldt for proving upon request a lemma used in an earlier proof of the case $B\le_T 0'$ of Theorem \ref{jada}, and G. Barmpalias, A.E.M. Lewis and M. Soskova for useful communications.}
	}
	\maketitle
	\begin{abstract}
		The low for random reals are characterized topologically, as well as in terms of domination of Turing functionals on a set of positive measure.
	\end{abstract}
	\tableofcontents
	\section{Introduction}

		A function
		$f:\omega\to\omega$ is
		\emph{uniformly almost everywhere (a.e.) dominating}
		if for measure-one many $X$, and all $g$ computable from $X$, $f$ dominates $g$. Such functions were first studied by Kurtz \cite{Kurtz1981} who showed that uniformly a.e.~dominating functions exist and that in fact  $0'$, the Turing degree of the halting problem, computes one of them. If we replace measure by category, there are no such functions, as is not hard to see. A few decades later Dobrinen and Simpson \cite{dob:04} made use of a.e. domination in Reverse Mathematics. They made a couple of fundamental conjectures that were promptly refuted in \cite{BKLS:05} and \cite{CGM}. In this article we strengthen the results of \cite{BKLS:05} to provide a characterization of a related concept, positive-measure domination, in terms of lowness for randomness. Conversely, we characterize low for random reals in terms of such domination. The following characterizations are already known. (We assume the reader is familiar with the definition of Martin-L\"of random reals and of prefix-free Kolmogorov complexity $K$.)

		\begin{thm}[Nies, Hirschfeldt, Stephan, Terwijn \cite{HNS},\cite{AM},\cite{NST}]
		The following are equivalent for $A\in 2^\omega$:
		\begin{itemize}
		 \item $A$ is low for random: {each Martin-L\"of random real is Martin-L\"of random relative to $A$.}
		 \item $A$ is $K$-trivial: $\exists c\forall n\, K(A\restrict n)\le K(\emptyset\restrict n)+c$.
		 \item $A$ is low for $K$: $\exists c\forall n\, K(n)\le K^A(n)+c$.
		 \item $\exists Z\ge_T A$, $Z$ is ML-random relative to $A$.
		 \item $A\le_T 0'$ and $\Omega$ is ML-random relative to $A$
		\end{itemize}
		\end{thm}

		The low for random reals induce a $\Sigma^0_3$ nonprincipal ideal in the Turing degrees bounded above by a low$_2$ $\Delta^0_2$ degree \cite{AM}, and have already found application to long-standing open problems in computability theory. Our characterizations in this paper are distinguished by not being couched in the language of randomness and Kolmogorov complexity. They do however refer to measure; it remains open whether a characterization purely in terms of domination or traces can be given such as that found for low for Schnorr random reals \cite{KNS}\cite{TZ}.

		The first main result of Section \ref{2} is Theorem \ref{AB}, which is a characterization of the low for random reals in terms of containment of effectively closed sets of positive measure. Building on this result, Theorem \ref{mainly} is a characterization of low for random reals in terms of positive-measure domination. Section \ref{3} contains, first, a characterization of the Turing degrees relative to which $0'$ is low for random,  in terms of positive-measure domination. Finally, with a view toward future research, we include a proof that there is a Turing functional that is universal for this kind of domination.

	\section{Low for random reals}\label{2}

		To obtain our topological characterization, we will pass first from a certain universal Martin-L\"of test (given in terms of $K$) to an arbitrary Martin-L\"of test, and then to an arbitrary open set of measure $<1$.

		\begin{thm}[Kraft-Chaitin Theorem \cite{Chaitin}]\label{KC}
		Suppose $\la n_k,\sigma_k\ra$, $k\in\omega$ is an $A$-recursive
		sequence, with $\sum_k 2^{-n_k}\le 1$. Then there exists a
		partial $A$-recursive prefix-free machine $M$ and a collection of strings $\tau_k$ with
		$|\tau_k|=n_k$ and $M(\tau_k)=\sigma_k$.
		\end{thm}

		\begin{df}[Chaitin]
		Let $A\in 2^\omega$. An \emph{information content measure relative to $A$} is a partial function $\hat K:2^{<\omega}\to\omega$ such that $$\sum_{\sigma\in 2^{<\omega}}2^{-\hat K(\sigma)}\le 1$$ and $\{\la\sigma,k\ra:\hat K(\sigma)\le k\}$ is r.e. in $A$.
		\end{df}

		\begin{lem}[Chaitin]\label{minimal}
		If $\hat K$ is an information content measure relative to a real $A$, then for all $n$,
		$K^A(n)\le \hat K(n)+\mc O(1)$.
		\end{lem}

		\begin{df}\label{remind}
		For any real $X$, let $S^A=\{S^A_n\}_{n\in\omega}$ where $S^A_n=\{X:\exists m\,\,K^A(X\restrict m)\le m-n\}$.
		\end{df}

		\begin{lem}\label{singapore}
		If $V^A$ is a Martin-L\"of test relative to $A$, then for each $n$ there exists $p$ such that
		$V^A_p\subseteq S^A_n$.
		\end{lem}
		\begin{proof}
		For each $m$, write $V^A_{2m}=\bigcup\{[\sigma_{m,k}]:k\in\omega\}$ where the function $f$ given by $f(m,k)=\sigma_{m,k}$ is computable and
		the sets $\{\sigma_{m,k}:k\ge 1\}$ are prefix-free. Define numbers
		$n_{m,k}=|\sigma_{m,k}|-m+1$ for $m,k\in\omega$. We have

		$$\sum_{m,k} 2^{-|n_{m,k}|}=\sum_m 2^{m-1} \sum_k 2^{-|\sigma_{m,k}|}=\sum_m 2^{m-1}\, \mu V_{2m} \le \sum_m
		2^{m-1} 2^{-2m} = 1.$$

		Hence by Theorem \ref{KC}, we have a partial $A$-recursive prefix-free machine $M$ and strings
		$\tau_{m,k}$ with $|\tau_{m,k}|=n_{m,k}$ and $M(\tau_{m,k})=\sigma_{m,k}$.
		Thus $K_M$, complexity based on the machine $M$, satisfies
		$K_M(\sigma_{m,k})\le n_{m,k}$, and so by Lemma \ref{minimal}, there is a constant $c$
		such that $K^A(\sigma_{m,k})\le n_{m,k}+c=|\sigma_{m,k}|-m+1+c$. This means that
		$V^A_{2m}\subseteq S^A_{m-c-1}$ for each $m$.

		Thus, given $n$, let $m=n+c+1$ and $p=2m$. Then $V^A_p=V^A_{2m}\subseteq
		S^A_{m-c-1}=S^A_n$, as desired.
		\end{proof}

		Schnorr \cite{241} showed that $S^A$ is a universal Martin-L\"of test relative to $A$.

		\begin{df}
		Let $n\ge 1$. Let $\Sigma^\mu_n$ denote the collection of all $\Sigma^0_n$
		classes of measure $<1$. The complement of a $\Sigma^\mu_n$ class is a $\Pi^\mu_n$ class. The complement of $U$ is denoted $\overline U$.
		The clopen subset of $2^\omega$ generated by $\sigma\in
		2^{<\omega}$ is denoted $[\sigma]$, and concatenation of strings
		is denoted by juxtaposition.

		If $U$, $V$ are open subsets of $2^\omega$ given by $U=\bigcup\{[\sigma]:\sigma\in \hat U\}$ and $V=\bigcup\{[\sigma]:\sigma\in \hat V\}$, where $U$ and $V$ are prefix-free sets of strings, then we define
		$$UV=\bigcup\{[\sigma\tau]:\sigma\in\hat U, \,\tau\in\hat V\}.$$
		This product depends on $\hat U$ and $\hat V$, not just on $U$ and $V$, so when considering a $\Sigma^0_1$ class $U$, we implicitly fix a suitable recursively enumerable set $\hat U$ for $U$.
		We define $U^n=U^{n-1}U$ where $U^1=U$. We can also think of this exponentiation as acting on a closed set $Q$, defining $Q^n$ via the equation $\overline{Q^n}=\overline Q^n$. It will be clear whether we are considering a set as open or closed.
		\end{df}
		\begin{lem}[Ku\v{c}era \cite{kucera}]\label{tonda}
		For each $\Pi^\mu_1(A)$ class $Q$ there is a computable function $f$ such that $\{\overline {Q^{f(n)}}\}_{n\in\omega}$, is a Martin-L\"of test relative to $A$.
		\end{lem}
		\begin{proof}
		Let $q>0$ be a rational number such that $\mu Q\ge q$. Let $P=\overline Q$. Then
		$\mu P^n=(\mu P)^n\le (1-q)^n$. Let $f$ be a computable function such that for all $k\in\omega$, $\mu P^{f(k)}\le 2^{-k}$. Let $V^A_k=P^{f(k)}$. Then $V^A$ is a Martin-L\"of test relative to $A$.
		\end{proof}

		\begin{lem}\label{WalMart}
		If $P$ is an open set such that
		$P^n$ is contained in a $\Sigma^\mu_1$ class for some $n\ge 2$, then $P$ itself
		is contained in a $\Sigma^\mu_1$ class.
		\end{lem}
		\begin{proof}
		We write $U|\sigma=\bigcup\{[\tau]: [\sigma\tau]\subseteq U\}$.
		Note that if $P$ is open then so is $P^2$. Hence by iteration, it suffices to consider the case $n=2$.
		So suppose $(\exists U)\,\,\,P^2\subseteq U\in\Sigma^\mu_1$. Case 1: $\exists\sigma$, $\mu (U|\sigma)<1$, $\sigma\in\hat P$.
		Then $P^2\cap[\sigma]=[\sigma]P$, the product of $[\sigma]$ and $P$.
		Then $P=([\sigma]P)|\sigma=(P^2\inter[\sigma])|\sigma=P^2|\sigma\subseteq U|\sigma\in\Sigma^\mu_1$.
		Case 2: Otherwise; so $\hat P\subseteq\left\{\sigma:\mu(U|\sigma)=1\right\}$. Fix a rational number $\epsilon>0$ such that $\mu U<1-\epsilon$, and let
		$V=\bigcup\left\{[\sigma]: \mu (U|\sigma)\ge 1-\epsilon \right\}$. Note that $V$ is $\Sigma^0_1$, contains $P$, and $\mu V<1$ because $(1-\epsilon)\mu V\le \mu U<1-\epsilon$.
		\end{proof}

		As usual, an $A$-random is a real that is Martin-L\"of random relative to $A$. If $A, B\in 2^\omega$ then $A$ is a \emph{tail} of $B$ if there exists $n$ such that $A(k)=B(n+k)$ for all $k\in\omega$.

		\begin{lem}[Ku\v{c}era \cite{kucera}]\label{antonin}
		For each $A\in 2^\omega$, each $\Pi^\mu_1(A)$ class contains a tail of each $A$-random real.
		\end{lem}
		\begin{proof}
		Let $Q$ be a $\Pi^\mu_1(A)$ class and suppose $X$ is $A$-random. Then by Lemma \ref{tonda}, there is an $m$ such that $X\in Q^m$. If $m=2$ then clearly, as $Q$ is closed, some tail of $X$ is an element of $Q$. If $m>2$, the result follows by iteration since each $Q^m$ is closed.
		\end{proof}

		\begin{thm}\label{AB}
		Let $A\in 2^\omega$. The following are equivalent:
		\begin{enumerate}
		\item Each 1-random real is $A$-random ($A$ is \emph{low for random} \cite{AM}).

		\item For each $\Pi^\mu_1$ class $Q$ consisting entirely of $1$-random reals, there exist $\sigma,n$ such that $Q\inter[\sigma]\ne\emptyset$ but $Q\inter S^A_n\inter [\sigma]=\emptyset$.

		\item For some $n$, $\overline{S^A_n}$ has a $\Pi^\mu_1$ subclass.
		\item For each $A$-Martin-L\"of
		test $V_n^A$, there exists an $n$ such that $\overline{V^A_n}$ has a $\Pi^\mu_1$ subclass.
		\item For each $\Pi^\mu_1(A)$ class $Q$ there exists an $n$ such that $Q^n$ has a $\Pi^\mu_1$ subclass.
		\item Each $\Pi^\mu_1(A)$ class has a $\Pi^\mu_1$ subclass.
		\item Some $\Pi^\mu_1(A)$ class consisting entirely of $A$-random reals has a $\Pi^\mu_1$ subclass.
		\item The class of $A$-random reals has a $\Pi^\mu_1$ subclass.
		\end{enumerate}
		\end{thm}
		\begin{proof}
		\noindent (1)$\Implies$(2): For this implication we use an argument of Nies and Stephan \cite{N}. Suppose $A$ is low for random but (2) fails.
		So there is a $\Pi^\mu_1$ class $Q$ consisting entirely of $1$-random reals, such that for all $\sigma,n$,  if $Q\inter[\sigma]\ne\emptyset$ then $Q\inter S^A_n\inter[\sigma]\ne\emptyset$.
		Let $\sigma_0=\lambda$, and $\sigma_{n+1}\succeq \sigma_n$, with $[\sigma_{n+1}]\subseteq S_n^A$ but $[\sigma_{n+1}]\inter Q\ne\emptyset$.
		Then $Y=\bigcup_{n\in\omega} \sigma_n$ is not $A$-random, but is $1$-random, since $Y\in Q$.
		\noindent (2)$\Implies$(3) Let $Q$ be as in (2), and let $n$, $\sigma$ be as guaranteed by (2) for $Q$. Then $Q\inter[\sigma]$ is the desired subclass. It has positive measure because no 1-random belongs to a $\Pi^0_1$ class of measure zero. (3)$\Implies$(4): Lemma \ref{singapore}.
		\noindent (4)$\Implies$(5): Let $Q$ be a $\Pi^\mu_1(A)$ class. By Lemma \ref{tonda}, $V^A_k=\overline {Q^{f(k)}}$ is a Martin-L\"of test relative to $A$ for some computable $f$. By (4), $Q^{f(m)}=\overline{V_m^A}\supseteq F$ for some $F\in\Pi^\mu_1$ and $m$; let $n=f(m)$. (5)$\Implies$(6): Lemma \ref{WalMart}.
		\noindent (6)$\Implies$(7): If $U^A$ is a universal Martin-L\"of test for $A$-randomness then we can let $Q=\overline {U_1}$.
		\noindent (7)$\Implies$(8): Since any class consisting entirely of $A$-randoms is contained in the class of all $A$-randoms.
		\noindent (8)$\Implies$(1): Suppose $X$ is $1$-random; we need to show $X$ is $A$-random. Let $F$ be a $\Pi^\mu_1$ subclass of the class of $A$-randoms. By Lemma \ref{antonin}, some tail of $X$ is an element of $F$. Hence a tail of $X$ is $A$-random, and thus $X$ itself is $A$-random.
		\end{proof}

		To characterize the low for random reals in terms of domination we first introduce some notation.
		We write Tot($\Phi$)=$\{X:\Phi^X$ is total$\}$ and $\varphi^X(n)=(\mu s)(\forall m<n)(\Phi_s^X(m)\downarrow\le s)$.
		Note that Tot($\Phi$) is a $\Pi^0_2$ class for each $\Phi$, and Tot($\Phi$)=Tot($\varphi$). The function $\varphi$ is the running time of $\Phi$, explicitly satisfying $\Phi^X(n)\le\varphi^X(n)$ for all $n$.
		Let $\Phi$ be a Turing functional and $B\in 2^\omega$.
		If there exists $f\le_T B$ such that for positive-measure many $X$, $\Phi^X$ is dominated by $f$, then we write $\Phi<B$. By $\sigma$-additivity this is equivalent to the statement that there exists $f\le_T B$ such that for positive-measure many $X$, $\Phi^X$ is \emph{majorized} by $f$. We also write $\Phi<B$ in the case that Tot($\Phi$) has measure zero.

		 \begin{lem}[implicit in \cite{dob:04}]\label{following}
		Let $B\in 2^\omega$ and let $\Phi$ be a Turing functional. Then $\varphi<B$ iff Tot($\Phi$) has a $\Pi^\mu_1(B)$ subclass.
		\end{lem}
		\begin{proof}
		First suppose $\varphi<B$, as witnessed by $f$. Then $\{X:\forall
		n\,\,\Phi^{X}_{f(n)}(n)\downarrow\}$ is a $\Pi^\mu_1(B)$ subclass of Tot($\Phi$).
		Conversely, let $F$ be a $\Pi^\mu_1(B)$ subclass of Tot($\Phi$). By
		compactness, $\{\varphi^{X}(n):X\in F\}$ is finite for each $n$, and $\{\la n,m\ra:\forall X(X\in F\to \varphi^{X}(n)<m\}$ is a $\Sigma^0_1(B)$ class. Hence by $\Sigma^0_1(B)$ uniformization there is a function
		$f\le_T B$ such that $\forall n\forall X(X\in
		F\to \varphi^{X}(n)<f(n))$; i.e., $f$ witnesses that $\varphi<B$.
		\end{proof}

		\begin{thm}\label{mainly}
		Let $A\in 2^\omega$. The following are equivalent:
		\begin{enumerate}
		\item $A$ is low for random.
		\item Each $\Pi^\mu_1(A)$ class has a $\Pi^\mu_1$
		subclass.
		\item (i) $A\le_T 0'$ and (ii) for each $\Phi$, if Tot($\Phi$) has a $\Pi^\mu_1(A)$ subclass then $\varphi<0$.
		\item (i) $A\le_T 0'$, and (ii) for each $\Phi$, if $\varphi<A$ then
		$\varphi<0$.
		\end{enumerate}
		\end{thm}
		\begin{proof}
		(1)$\Iff$(2) was shown in Theorem \ref{AB}. (2)$\Implies$(3): Nies \cite{AM} shows that if $A$ is low for random then $A\le_T 0'$.
		Suppose Tot($\Phi$) has a $\Pi^\mu_1(A)$ subclass $Q$. By (2), Tot($\Phi$)
		has a $\Pi^\mu_1$ subclass $F$. By Lemma \ref{following},
		we are done.
		(3)$\Implies$(2): Suppose (3) holds and suppose $Q$ is a $\Pi^\mu_1(A)$ class. Pick $\Psi$ such that $Q=\{X:\Psi^{X\oplus A}(0)\uparrow\}$. Since $A\le_T 0'$, $A=\lim_s A_s$, the limit of a computable approximation. Let $\Phi^{X}(s)=\mu t>s(\Psi_t^{X\oplus A_t}(0)\uparrow)$. Then $Q=$Tot($\Phi$). Applying (3) to this $\Phi$,  we have $\varphi<0$ and so by Lemma \ref{following} we are done.
		(3)$\Iff$(4) is immediate from Lemma \ref{following}.
		\end{proof}

	\section{Positive-measure domination}\label{3}

		In \cite{dob:04} it was asked whether the Turing degrees $A$ of uniformly a.e. dominating functions are characterized by either of the inequalities $A\ge 0'$ and $A'\ge_T 0''$. The case $A\ge 0'$ was refuted by a direct construction in \cite{CGM}. The case $A'\ge_T 0''$ was refuted in \cite{BKLS:05} using precursors to the results presented here. Namely, the dual of property 4(ii) above is $\forall\varphi(\varphi<A)$ or equivalently $\forall\varphi(\varphi<0'\to\varphi<A)$. Relativizing our proofs gives that this is equivalent to: $0'$ is low for random relative to $A$. If we restrict ourselves to $A\le_T 0'$, then by \cite{AM} this implies $A'\ge_{tt}0''$, which is strictly stronger than $A'\ge_T 0''$. We do not know whether the assumption $A\le_T 0'$ is necessary for either of the conclusions $A'\ge_{tt}0''$, $A'\ge_T 0''$.

		We say that $A$ is \emph{positive-measure dominating} if for each $\Phi$, $\Phi<A$.
		If each $B$-random real is $A$-random then we write $A\le_{LR}B$ (A is low for random relative to $B$) following \cite{AM}. We write $\Phi^A$ for the functional $X\mapsto \Phi^{A\oplus X}$.

		\begin{lem}\label{refreq}
		Let $A\in 2^\omega$ and $\mc C\subseteq 2^\omega$. Then $\mc C$ is a $\Pi^0_2(A)$ class iff $\mc C$ is Tot($\Phi^A$) for some Turing functional $\Phi$.
		\end{lem}
		\begin{proof}
		Suppose $\mc C$ is a $\Pi^0_2(A)$ class, i.e. $\mc C=\{X:\forall y\exists s R(y,s,A,X)\}$ where $R$ is a formula in the language of second-order arithmetic all of whose quantifiers are first-order and bounded. Then we can let $\Phi^{A\oplus X}(y)=\mu s(R(y,s,A,X))$. Conversely, Tot($\Phi^A)=\{X:\forall y\exists s(\Phi_s^{A\oplus X}(y)\downarrow)\}$.
		\end{proof}

		\begin{thm}\label{jada}
		Let $B\in 2^\omega$. Then $0'$ is low for random relative to $B$ iff $B$ is positive-measure dominating.  \end{thm}
		\begin{proof}
		This is the special case $A=0$ of the fact that for each $A, B\in 2^\omega$, the following are equivalent:
		\begin{enumerate}
		\item $A'\le_{LR}A\oplus B$.
		\item Each $\Pi^\mu_1(A')$ class has a $\Pi^\mu_1(A\oplus B)$ subclass.
		\item Each $\Pi^\mu_2(A)$ class has a $\Pi^\mu_1(A\oplus B)$ subclass.
		\item $\forall\Phi$, if Tot($\Phi^A$) has positive measure then it has a $\Pi^\mu_1(A\oplus B)$ subclass.
		\item $\forall\Phi(\varphi^A<A\oplus B$)
		\end{enumerate}
		The equivalences are proved as follows. (1)$\Iff$(2): Relativization of Theorem \ref{AB} gives: $A\le_{LR}B$ iff each $\Pi^\mu_1(A)$ class has a $\Pi^\mu_1(B)$ subclass. (3)$\Iff$(4): Lemma \ref{refreq}. (4)$\Iff$(5): Relativization of Lemma \ref{following}. (2)$\Iff$(3): Let $A\in 2^\omega$. $A'$ is uniformly a.e.~dominating relative to $A$, hence $A'$ is positive-measure dominating relative to $A$. Hence by putting $B=A'$ in (3)$\Iff$(5), each $\Pi^\mu_2(A)$ class has a $\Pi^\mu_1(A')$ subclass.

		\end{proof}

		\subsection*{Universal functionals \footnote{The published version of this section contained a mistake, which has here been corrected.}}

		Suppose $\Phi_i$, $i\in\omega$ are all the Turing functionals.
		As observed in \cite{CGM}, the functional $\Psi$ given by $\Psi^{0^i1X}=\Phi_i^X$ is \emph{universal} for uniform a.e. domination, in the sense that any function that dominates $\Psi$ on almost every $X$, is a uniformly a.e. dominating function. As $\Psi<0$, $\Psi$ is not universal for positive-measure domination; however, the following functional is.

		Fix $c\in\omega$. Let $U^Y$ be universal among prefix-free Turing machines with oracle $Y$.

		Then
		\begin{equation}
		\tag{0}(\forall n)(K^{0'}(X\restrict n)\ge n-c)
		\end{equation}
		is equivalent to
		\begin{equation}\tag{1} (\forall n)(\forall\sigma\in 2^{<n-c})( \neg (U^{0'}(\sigma)=X\restrict n)).
		\end{equation}

		Let $\upsilon_t(\sigma)$ be the use of $0'$ in the computation $U^{0'[t]}_t(\sigma)$. (If the latter is undefined then so is the former.)

		This is again equivalent to
		\begin{equation}\tag{2} (\forall n)(\forall\sigma\in 2^{<n-c})(\forall s)(\exists t\ge s)
		\end{equation}
		\begin{equation}
		\notag \neg(U^{0'[t]}_t(\sigma)=X\restrict n \text{ and }0'[t]\restrict \upsilon_t(\sigma)=0'[t-1] \restrict \upsilon_t(\sigma)).
		\end{equation}

		(2)$\Implies$(1): Suppose $\neg$(1), so that actually $U^{0'}(\sigma)=X\restrict n$. Let $s$ be such that $0'$ has stabilized up to the use of $U^{0'}(\sigma)$ by stage $s-1$. Then for all $t\ge s$, $U^{0'[t]}_t(\sigma)=X\restrict n$, and $\neg$(2) follows.

		(1)$\Implies$(2): Suppose $\neg$(2), as witnessed by $n$, $\sigma$, and $s$. Thus, for all $t\ge s$ we have $U^{0'[t]}_t(\sigma)=X\restrict n))$, and $0'$ never changes below the use of $U^{0'}(\sigma)$ after stage $s$. Thus $U^{0'}(\sigma)=X\restrict n$.

		Let $\Xi_c^X(n,\sigma,s)$ be the least stage $t\ge s$ at which $X$ looks like it is 2-random, with constant $c$, in the sense that
		$$|\sigma|< n-c\to \neg(U^{0'[t]}_t(\sigma)=X\restrict n \text{ and }0'[t]\restrict \upsilon_t(\sigma)=0'[t-1] \restrict \upsilon_t(\sigma)).$$
		Then $\Xi_c^X$ is total iff (0) holds. Thus $\Xi_c$ is total for positive-measure many $X$, all of which are 2-randoms.
		 The running time $\xi_c$ of $\Xi_c$ is universal for positive-measure domination in the following sense.

		\begin{thm}
		The class $\{A: A$ is positive-measure dominating$\}$ is $\Sigma^0_3$. In fact, for each $A\in 2^\omega$ and $c\in\omega$, $A$ is positive-measure dominating iff $\xi_c<A$.
		\end{thm}
		\begin{proof}
		Suppose $\xi_c<A$.
		By Lemma \ref{following},
		Tot($\Xi_c$) has a $\Pi^\mu_1(A)$ subclass.

		The complement of Tot($\Xi_c$)
		is $\{X: \exists n\, K^{0'}(X\restrict n)<n-c\}$ which is open. Hence Tot($\Xi_c$) is
		closed and is in fact a $\Pi^\mu_1(0')$ class.

		Thus: Some $\Pi^\mu_1(0')$ class consisting entirely of $0'$-randoms has a $\Pi^\mu_1(A)$ subclass.
		By Theorem \ref{AB} (7) relativized,  $0'\le_{LR}A$, and so by Theorem \ref{jada}, $A$ is positive-measure dominating.
		\end{proof}

	\bibliographystyle{amsplain}
	\bibliography{PAMSwithHyperlinks-arxiv}
\end{document}